\documentclass[12pt,reqno]{amsart}
\usepackage{graphics}
\usepackage{graphicx} 
\usepackage{psfrag} 
\usepackage{amssymb,a4wide,amsthm,amsmath,amsfonts}
\usepackage{mathrsfs}
\usepackage{color}
\usepackage{enumerate}

\theoremstyle{plain}
\newtheorem{theorem}{Theorem}[section]
\theoremstyle{remark}

\theoremstyle{plain}
\newtheorem{corollary}[theorem]{Corollary}
\newtheorem{lemma}[theorem]{Lemma}
\newtheorem{proposition}[theorem]{Proposition}
\newtheorem{definition}[theorem]{Definition}

\newtheorem{claim}{Claim}

\numberwithin{equation}{section}

\newcommand{\R}{\mathbb{R}}
\newcommand{\N}{\mathbb{N}}

\newcounter{probcount}

\newcommand{\Problem}{\addtocounter{probcount}{1}
                       {\bf Problem \arabic{probcount}. }}

\begin{document}
\title{Large structures made of nowhere $L^q$ functions}

\author[ Sz. G\l \c ab]{ Szymon G\l \c ab}
\address{Institute of Mathematics, Technical University of
\L\'od\'z, W\'olcza\'nska 215, 93-005 \L\'od\'z,   Poland}
\email{szymon.glab@p.lodz.pl}
\thanks{The first named author was supported by the Polish Ministry of Science and Higher Education Grant No. N N201 414939 (2010-2013).}

\author[P. L. Kaufmann]{ Pedro L. Kaufmann }
\address{Instituto de matem\'atica e estat\'istica, Universidade de S\~ao Paulo, Rua do Mat\~ao, 1010, CEP 05508-900, S\~ao Paulo,  Brazil}
\email{plkaufmann@gmail.com}
\thanks{The second named author was supported by CAPES, Research Grant PNPD 2256-2009, and by the Institut de Math\'ematiques de Jussieu.}

\author[L. Pellegrini]{ Leonardo Pellegrini}
\address{Instituto de matem\'atica e estat\'istica, Universidade de S\~ao Paulo, Rua do Mat\~ao, 1010, CEP 05508-900, S\~ao Paulo,  Brazil}
\email{leonardo@ime.usp.br}

\subjclass[2010]{46G12, 15A03}
\keywords{nowhere $L^q$ functions, spaceability, algebrability}

\begin{abstract}
We say that a real-valued function $f$ defined on a positive Borel measure space $(X,\mu)$ is nowhere $q$-integrable if, for each nonvoid open subset $U$ of $X$, the restriction $f|_U$ is not in $L^q(U)$. When $(X,\mu)$ satisfies some natural properties, we show that certain sets of functions defined in $X$ which are $p$-integrable for some $p$'s but nowhere $q$-integrable for some other $q$'s ($0<p,q<\infty$) admit a variety of large linear and algebraic structures within them. 
The presented results answer a question from Bernal-Gonz\'alez,  improve and complement recent spaceability and algebrability results from several authors and motivates new research directions in the field of spaceability.

\end{abstract}
\maketitle

\section{Introduction}

This work is a contribution to the study of large linear and algebraic structures within essentially nonlinear sets of functions which satisfy special properties; the presence of such structures is often described using the terminology \emph{lineable}, \emph{algebrable} and \emph{spaceable}.  Recall that a subset $S$ of a topological vector space $V$ is said to be \emph{lineable} (respectively, \emph{spaceable}) if $S\cup\{0\}$ contains an infinite dimensional vector subspace (respectively, a \emph{closed} infinite dimensional  vector subspace) of $V$. Though results in this field date back to the sixties\footnote{In \cite{gurrus}, Gurariy showed that there  exists in $C([0,1])$ a closed infinite-dimensional subspace consisting, except for the null function, only on nowhere differentiable functions - see also \cite{gur} for a version in english.}, this terminology was not introduced until recently: it first appeared in unpublished notes by Enflo and Gurariy and was firstly published in \cite{ags}. We should mention that Enflo's and Gurariy's unpublished notes were completed in collaboration with Seoane-Sep\'ulveda and will finally be published in \cite{egs}. It is current to say also that $S$ is \emph{dense-lineable} if $S\cup\{0\}$ contains an \emph{dense} infinite dimensional vector subspace of $V$. The adjective \emph{maximal} is often added to \emph{dense-lineable} or \emph{spaceable} when the corresponding space contained in $S\cup\{0\}$ has the same dimension of $V$. We propose in Section \ref{secvsp} a notion of spaceability which is more restrictive than the ``maximal'' spaceability in terms of dimension. For this reason, we choose to use the notation \emph{maximal-dimension spaceable}, \emph{maximal-dimension lineable} and so on when the maximality concerns dimension of the subspace found in $S\cup\{0\}$. 

The term \emph{algebrability} was introduced later in \cite{as}; if $V$ is a linear algebra, $S$ is said to be $\kappa$-algebrable if $S\cup\{0\}$ contains an infinitely generated algebra, with a \emph{minimal} set of generators of cardinality $\kappa$ (see \cite{as} for details). We shall work with a strenghtened notion of $\kappa$-algebrability, namely, \emph{strong $\kappa$-algebrability}. The definition follows:
\begin{definition}
We say that a subset $S$ of an algebra $\mathcal{A}$ is \emph{strongly $\kappa$-algebrable}, where $\kappa$ is a cardinal number, if there exists a $\kappa$-generated free algebra $\mathcal{B}$ contained in $S\cup\{0\}$. 
\end{definition}

We recall that, for a  cardinal number $\kappa$, to say that an algebra $\mathcal{A}$ is a \emph{$\kappa$-generated free algebra}, means that there exists a subset $Z=\{z_\alpha:\alpha<\kappa\}\subset \mathcal{A}$ such that any function $f$ from $Z$ into some algebra $\mathcal{A'}$ can be uniquely extended to a homomorphism from $\mathcal{A}$ into $\mathcal{A'}$. The set $Z$ is called a \emph{set of free generators} of the algebra $\mathcal{A}$. If $Z$ is a set of free generators of some subalgebra $\mathcal{B} \subset \mathcal{A}$, we say that $Z$ is a set of free generators \emph{in} the subalgebra $\mathcal{A}$. If $\mathcal{A}$ is \emph{commutative}, a subset $Z=\{z_\alpha:\alpha<\kappa\} \subset \mathcal{A}$ is a set of free generators in $\mathcal{A}$ if for each polynomial $P$ and for any $z_{\alpha_1},z_{\alpha_2},\dots,z_{\alpha_n} \in Z$ we have 
$$
P(z_{\alpha_1},z_{\alpha_2},\dots,z_{\alpha_n})=0 \mbox{ if and only if } P=0. 
$$
The definition of strong $\kappa$-algebrability was introduced in \cite{bsz}, though in several papers, sets which are shown to be algebrable are in fact strongly algebrable, and that is seen clearly by the proofs. See \cite{as}, \cite{bq} and \cite{pathos}, among others. 
Strong algebrability is in effect a stronger condition than algebrability: for example, $c_{00}$ is $\omega$-algebrable in $c_0$ but it is not strongly $1$-algebrable (see \cite{bsz}).


\subsection{Results on large structures of non-integrable functions: recent and new}

Our object of study will be the quasi-Banach spaces $L^p(X,\mathcal{M},\mu)$. For a clear notation, when there cannot be any confusion or ambiguity, we shall write $L^p$, $L^p(X,\mu)$ or $L^p(X)$ instead of $L^p(X,\mathcal{M},\mu)$. Our main focus will be on functions which are $p$-integrable but not $q$-integrable, for some $0<p,q\leq\infty$, and specially on functions which are $p$-integrable but \emph{nowhere} $q$-integrable. The notion of nowhere-$q$-integrability we consider is connected to open sets: 

\begin{definition}

Let $0<q\leq\infty$. A scalar-valued function $f$ defined on a Borel measure space $X$ is said to be \emph{nowhere $q$-integrable} (or \emph{nowhere $L^q$}) if, for each nonvoid open subset $U$ of $X$, the restriction $f|_U$ is not in $L^q(U)$.

\end{definition}

In our context it would be pointless to substitute ``for each nonvoid open subset $U$ of $X$'' by ``for each Borel subset $U$ of positive measure of $X$'' in the definition above; the reason is that, if $0<p,q\leq \infty$ and $f\in L^p(X)$, there is always a Borel subset of $X$ with positive measure and on which $f$ is $q$-integrable. This follows from a simple argument (see e.g. the final remarks in \cite{b-g}). Of course, not all Borel measure spaces $(X,\mu)$ admit $L^p$-nowhere-$L^q$ functions, but there is a large class of such spaces which admit plenty of such functions, as we will see.  \\




Let us start by mentioning some recent results and open questions on large structures within sets of functions which are $p$-integrable but not $q$-integrable. For a survey on the evolution of the results in this direction, we recommend \cite{bcfps}.

\begin{theorem}[Bernal-Gonz\'alez, Ord\'o\~nez Cabrera \cite{bc}]

Let $(X,\mathcal{M},\mu)$ be a measure space, and consider the conditions

($\alpha$) $\inf\{\mu(A):A\in \mathcal{M}, \mu(A)>0\}=0$, and

($\beta$) $\sup\{\mu(A):A\in \mathcal{M}, \mu(A)<\infty\}=\infty$. \\
Then the following assertions hold: 
\begin{enumerate}
\item if $1\leq p <\infty$, then $L^p\setminus \cup_{q>p}L^q$ is spaceable if and only if ($\alpha$) holds;
\item if $1<  p \leq\infty$, then $L^p\setminus \cup_{q<p}L^q$ is spaceable if and only if ($\beta$) holds;
\item if $1< p <\infty$, then $L^p\setminus \cup_{q\neq p}L^q$ is spaceable if and only if both ($\alpha$) and ($\beta$) hold; 
\item if $1< p <\infty$ and $L^p$ is separable, then $L^p\setminus \cup_{q< p}L^q$ is maximal-dimension dense-lineable if and only if ($\beta$) holds. 

\end{enumerate}

\label{frankenstein}
\end{theorem}

Note that any of the conditions ($\alpha$), or ($\beta$), or ($\alpha$) \emph{and} ($\beta$) is enough to guarantee the existence of nowhere $q$-integrable functions in $L^p$ (just note that, if $X$ contains an open singleton $\{x\}$ of positive measure, then each function from $L^p(x)$ is $q$-integrable in $\{x\}$, for all $q$). Bernal-Gonz\'alez et. al. use the convenient terminology \emph{`(left, right) strict order integrability'} when a function is $p$-integrable but not $q$-integrable for $q\neq p$ ($q<p$, $q>p$). We refer to \cite{bcfps} for improvements on item (2) of Theorem (\ref{frankenstein}) above. 
And in \cite{bfps} there is a version of that same item which includes quasi-Banach spaces:

\begin{theorem}[Botelho, F\'avaro, Pellegrino, Seoane-Sep\'ulveda \cite{bfps}]

$L^p[0,1]\setminus\cup_{q>p}L^q[0,1]$ is spaceable for every $p>0$. 

\label{quasi01}
\end{theorem}

When it comes to nowhere integrable functions, Bernal-Gonz\'alez gave the first initial result: 

\begin{theorem}[Bernal-Gonz\'alez \cite{b-g}]

Let $(X,\mathcal{M},\mu)$ be a measure space such that $X$ is a Hausdorff first-countable separable locally compact perfect topological space and that $\mu$ is a positive Borel measure which is continuous, regular and has full support. Let $1\leq p < \infty$.Then the set 
\begin{eqnarray}
\{ f\in L^p: f \mbox{ is nowhere $q$-integrable, for each $q>p$} \}
\label{gonzalezdense}
\end{eqnarray}
is dense in $L^p$.
\label{thgondens}
\end{theorem}

It is clear that $\mu$ having full support (that is, $\mu(U)>0$ for every nonvoid open subset $U\subset X$) is a necessary condition for the existence of nowhere $q$-integrable functions. Based on the above result, Bernal-Gonz\'alez rose the following question:\\

\Problem Is the set (\ref{gonzalezdense}) lineable/maximal-dimension lineable/dense-lineable?\\ 

It is quite natural to seek for other large structures within (\ref{gonzalezdense}). \\

The authors of this work have also presented some results on large structures of  nowhere integrable functions, and among them we mention the following: 

\begin{theorem}[G\l\c ab, Kaufmann, Pellegrini \cite{gkp}]

The set of nowhere essentially bounded functions in $L^1[0,1]$ is
\begin{enumerate}
\item spaceable and
\item strongly $\mathfrak{c}$-algebrable. 
\end{enumerate}

\label{thgkp}
\end{theorem}

In this landscape, we present a few new results which solve Problem 1 and, under quite mild conditions on the measure space where our functions are defined, complement/generalize the results mentioned above. We summarize these results in Theorem \ref{mainth} below. 

\begin{theorem}

Suppose that $X$ is a topological space admitting a countable $\pi$-base (that is, a family $(U_n)_n$ of nonvoid open subsets of $X$ such that, for each nonvoid open subset $A$ of $X$, $U_j\subset A$ for some $j$) and that $\mu$ a positive Borel measure on $X$. Let $0<p<\infty$ and consider the sets 
\begin{eqnarray*}
S_p(X)\doteq S_p \doteq \{f\in L^p: f \mbox{ is nowhere }L^q \mbox{, for each } p<q\leq\infty\},
\label{defSp}
\end{eqnarray*}
$$
S'_p\doteq S_p\setminus \cup_{0<q<p} L^q, \mbox{ and}
$$
$$
\mathcal{G} \doteq \left\{f\in \bigcap_{0<q<\infty} L^q: \mbox{ $f$ is nowhere }L^\infty\right\}.
$$
Then we have the following:
\begin{enumerate}[(a)]
\item if $\mu$ is atomless, outer regular and has full support, then $S_p\cup \{0\}$ contains a $\ell_p$-isometric subspace of $L^p$, which is in addition complemented if $p\geq 1$; 
\item if $\mu$ infinite and $\sigma$-finite, then $L^p\setminus \cup_{0<q<p}L^q$ contains a $\ell_p$-isometric subspace of $L^p$, which is in addition complemented if $p\geq 1$;
\item if $\mu$ is atomless, infinite, outer regular and has full support, then $S'_p \cup \{0\}$ contains a $\ell_p$-isometric subspace of $L^p$, which is in addition complemented if $p\geq 1$;
\item if $\mu$ is atomless, outer regular and has full support, then $S_p$ is maximal-dimension dense-lineable;
\item if $\mu$ is atomless, outer regular and has full support, then $\mathcal{G}$ is strongly $\mathfrak{c}$-algebrable. 

\end{enumerate}

\label{mainth}
\end{theorem}



See Section \ref{secfinal} for comments on the choice of working with $\pi$-bases instead of the more usual bases of open sets. In addition to Theorem \ref{mainth} we also prove that, for a special classes of positive Borel measure spaces, $S_p$ contains an isomorphic copy of $\ell_2$ (see Theorems \ref{thl2} and \ref{thl201}, and Corollary \ref{coroll2}). This motivates a new investigation direction concerning spaceability (see Section \ref{secvsp}). \\

{\bf Remark 1.} Referring to items (a)--(c), it is worth recalling that for $p<1$, $L^p$ contains no complemented copy of $\ell_p$. This is easily seen if one recalls that, for $p<1$, $\ell_p$ admits nontrivial continuous linear forms (e.g. the evaluation functionals), while every nontrivial linear form on $L^p$ is discontinuous. \\


{\bf Remark 2.} In any measurable space which admits a set of strictly positive finite measure (in particular for $(X,\mu)$ under the conditions in (e)) and $0<p<q<\infty$, the set of $L^p$ functions which are not $L^q$ is \emph{not} algebrable; to see this, just note that if $f$ is $p$- integrable but not $q$-integrable on some set of finite measure $U$, then $f^n$ is not $p$-integrable if we choose a large enough power $n$. There is therefore no hope in looking for algebraic structures of strict-order integrable functions in many cases. One exception is given by:

\begin{theorem}[Garc\'ia-Pacheco, P\'erez-Eslava, Seoane-Sep\'ulveda \cite{gps}]

If $(X,\mathcal{M},\mu)$ is a measure space in which there exists and infinite family of pairwise disjoint measurable sets $A_n$ satisfying $\mu(A_n)\geq \epsilon$ for some $\epsilon >0$, then 
$$
L^\infty \setminus \cap_{p=1}^\infty L^p
$$ 
is spaceable in $L^\infty$ and algebrable. 
\label{thgps}
\end{theorem}
Note that Theorem \ref{mainth}(e) complements, in some sense, the algebrability part of Theorem \ref{thgps}. \\


{\bf Remark 3.} Theorem \ref{mainth} relates to what was mentioned previously in the following way: 
 \begin{itemize}
\item (a) generalizes Theorem \ref{quasi01}, Theorem \ref{thgkp}(1) and, under our assumptions, also Theorem \ref{frankenstein}(1);
\item it is not hard to adapt Theorem \ref{frankenstein}(2) for $p<1$ and to see that the space guaranteeing the spaceability can be isometric to $\ell_p$ and complemented in case $p\geq 1$; since condition ($\beta$) from Theorem \ref{frankenstein} is milder that the conditions in (b), it follows that (b) does not really add much. But the construction in the proof we present is used to prove also (c), thus we include (b) for completeness and clearness;
\item under our assumptions, (c) improves Theorem \ref{frankenstein}(3); 
\item (d) improves Theorem \ref{thgondens} and gives a positive answer to Bernal-Gonz\'alez's Problem 1;
\item (e) improves Theorem \ref{thgkp}(2). \\
\end{itemize}

The remaining sections will be organized as follows. 
In Section \ref{secsp} we will prove Theorem \ref{mainth}(a)--(c), that is, its spaceability part. In Section \ref{secvsp}, we introduce the notion of \emph{$V$-spaceability} (see Definition \ref{defvspc}) and provide the results on $\ell_2$-spaceability. Section \ref{secdens} contains the proof of the dense-lineability result (Theorem \ref{mainth}(d)), and Section \ref{secalg} is on the algebrability result (Theorem \ref{mainth}(e)). In Section \ref{secfinal} we briefly discuss conditions on positive Borel measure spaces under which there exist, or not, functions $p$-nowhere-$q$ integrable in the corresponding $L^p$ spaces. We include related open problems throughout the text. 



\section{Spaceability: proof of Theorem \ref{mainth}(a)--(c)}
\label{secsp}

Recall the following standard result from functional analysis on Banach spaces:

\begin{theorem}

Suppose that $(X,\mu)$ is a Borel measure space. Let $1\leq p <\infty$, and suppose that $(f_n)$ is a sequence of norm-one, disjointly supported
functions in $L^p(\mu)$. Then $(f_n)$ is a complemented basic sequence isometrically
equivalent to the canonical basis of $\ell_p$. 

\label{stand}
\end{theorem}

It is not hard to see that the same holds for $0<p<1$, though the complementability is lost, as we previously pointed out. Our strategy to prove Theorem \ref{mainth}(a)--(c) will be to find sequences of norm-one, disjointly supported functions in $S_p$, $L^p\setminus\cup_{0<q<p}L^q$ and $S'_p$, under the corresponding assumptions. 



\begin{lemma}
Let $X$ be a topological space with a countable $\pi$-base. 
Suppose that $\mu$ is an atomless and outer-regular positive Borel measure on $X$ with full support. Let $U$ be an open set such that $\mu(U)>0$ and let $\varepsilon\in(0,1)$. Then there is a nowhere-dense Borel subset $N$ of $U$ such that $\mu(N)>\mu(U)\varepsilon$. 
\label{fact1}
\end{lemma}

\begin{proof} Let $(U_n)$ be a $\pi$-base of $U$. Since $\mu$ is atomless, there are Borel sets $B_n\subset U_n$ such that $\mu(B_n)<\varepsilon\mu(U)/2^n$. Since $\mu$ is outer-regular, there are open set $V'_n\supset B_n$ with $\mu(V_n')<\varepsilon\mu(U)/2^n$. Let $V_n=V_n'\cap U$ and put $V=\bigcup_{n}V_n$. Then $\mu(V)<\varepsilon\mu(U)$ and $V$ is a dense open subset of $U$. Therefore $N=U\setminus V$ is nowhere dense subset of $U$ with measure greater than $\mu(U)\varepsilon$. 
\end{proof}

\begin{lemma}

Suppose that $\mu$ is an atomless positive Borel measure on $X$ with full support. Let $A$ be a measurable set in $X$ such that $\mu(A)>0$ and let $(a_n)$ be a sequence in $(0,+\infty )$. Then there is a sequence $(A_n)$ of pairwise disjoint measurable subsets of $A$ such that $0<\mu(A_n)< \infty$ and $\mu(A_{n+1}) \leq a_n \mu(A_n)$.

\label{fact2}
\end{lemma}

\begin{proof}  We may assume that $a_n\leq 1/2$ for all $n$. Since $\mu$ is atomless, there is a Borel set $A_1\subset A$ such that 
$$
0<\mu(A_1)<\frac12 \mu(A).
$$ 
Likewise, there is Borel set $A_2\subset A\setminus A_1$ such that 
$$
0<\mu(A_2)<a_1\mu(A_1) \leq \frac12 \mu(A_1).
$$ 
Proceeding this way,  we can find inductively $A_{n}\subset A\setminus\bigcup_{k<n} A_k$ such that 
$$
0<\mu(A_n)<a_{n-1}\mu(A_{n-1})\leq \frac12 \mu(A_{n-1});
$$ this is possible since $\mu(A\setminus\bigcup_{k<n}A_k)>0$.
\end{proof}

\begin{lemma}

Suppose that $\mu$ is an atomless positive Borel measure on $X$ with full support. Then for any given Borel set $A$ in $X$ such that $\mu(A)>0$ there is a norm-one, $A$-supported function $h_A$ in $L^p\setminus\bigcup_{q>p}L^q$. 

\label{hA}
\end{lemma}

\begin{proof} Let $A\subset X$ be measurable and $\mu(A)>0$; by Lemma \ref{fact2} there exists a family $\{A_{n,m}:n,m\in\N\}$ of pairwise disjoint subsets of $A$ of positive measure such that $\mu(A_{n,m+1})\leq \frac12 \mu(A_{n,m})$. Let $(r_n)$ be a strictly decreasing sequence of real numbers tending to $p$. Put 
$$
h_n=\sum_{m=1}^\infty a_{n,m}\chi_{A_{n,m}},
$$    
where $a^{r_n}_{n,m}\mu(A_{n,m})=1/m$. Then $||h_n||_{r_n}=\infty$ and 
$$
\int_X|h_n|^pd\mu=\sum_{m=1}^\infty a_{n,m}^p\mu(A_{n,m})=\sum_{m=1}^\infty \frac{1}{a_{n,m}^{r_n-p}m}.
$$ 
Since
$$\limsup_{m\to\infty}\frac{\frac{1}{a_{n,m+1}^{r_n-p}(m+1)}}{\frac{1}{a_{n,m}^{r_n-p}m}}=
\limsup_{m\to\infty} \left(\frac{\mu(A_{n,m+1})}{\mu(A_{n,m})}  \right)^{\frac{r_n-p}{r_n}} \leq 
\left(\frac12 \right)^{\frac{r_n-p}{r_n}}<1, 
$$ 
then by the ratio test for series we obtain that $h_n\in L^p$.
Put 
$$
h_A=\sum_{n=1}^\infty \frac{h_n}{||h_n||2^n}.
$$
Then $h_A\in L^p\setminus\bigcup_{q>p}L^q$ and $\|h_A\|=1$. 
\end{proof}

{\it Proof of Theorem \ref{mainth}(a). }Let $(U_n)$ be a $\pi$-base of $X$. Since $\mu$ is atomless and outer-regular, we may assume that $\mu(U_n)<\infty$ for each $n$. (Indeed, suppose that $\mu(U_n)=\infty$. Hence $U_n\neq\emptyset$ and there is $x\in U_n$. Since $\mu$ is atomless, then $\mu(\{x\})=0$.  By the outer-regularity of $\mu$, there is an open neighborhood $V$ of $x$ with arbitrarily small $\mu$-measure. Since $\mu$ does not vanish on open sets, then $0<\mu(V\cap U_n)<\infty$. We may replace $U_n\cap V$ with $U_n$.) 

By Lemma \ref{fact1}, there is a nowhere dense Borel set $N_1\subset U_1$ with $0<\mu(N_1)<\frac12$. Since $N_1$ is nowhere dense we can find a nonempty open set $U\subset U_2\setminus N_1$, and again by Lemma \ref{fact1} there is a nowhere dense Borel set $N_2\subset U\subset U_2$ with $0<\mu(N_2)<\frac{1}{2^2}$. We can then inductively define a pairwise disjoint sequence of nowhere dense Borel sets $(N_n)$ such that $N_n\subset U_n$ and $0<\mu(N_n)<1/2^n$. Decompose each $N_n$ into $\mu$-positive and pairwise disjoint Borel sets $N_{n,m}$. For each $n,m$ there exists, by Lemma \ref{hA}, a norm-one, $N_{n,m}$-supported function $h_{N_{n,m}}$ in $L^p\setminus\bigcup_{q>p}L^q$. If we put
$$
f_m=\sum_{n=1}^\infty \frac{h_{N_{n,m}}}{2^n}, 
$$
then $(f_m)$ will form a norm-one basic sequence of elements from  $S_p$ with pairwise disjoint supports, and by Theorem \ref{stand} our proof is concluded. \hfill{$\Box$}

\begin{lemma}

Suppose that $\mu$ is an infinite and $\sigma$-finite positive Borel measure on $X$. Then for any given Borel set $B\subset X$ of infinite measure, there exists a function $g_B\in L^p\setminus\bigcup_{q<p} L^q$ which is zero outside of $B$.   

\label{lemmapart2}
\end{lemma}
\begin{proof} Let $B\subset X$ be Borel of infinite measure, and let $\{B_{n,m}:n,m\in\N\}$ be a family of pairwise disjoint subsets of $B$ of positive finite measure such that $2\mu(B_{n,m})\leq \mu(B_{n,m+1})$. Let $(r_n)$ be a strictly increasing sequence of (strictly positive) real numbers tending to $p$. Put 
$$
g_n=\sum_{m=1}^\infty b_{n,m}\chi_{B_{n,m}},
$$    
where $b^{r_n}_{n,m}\mu(B_{n,m})=1/m$. Then  
$$
\int_X|g_n|^pd\mu=\sum_{m=1}^\infty b_{n,m}^p\mu(B_{n,m})=\sum_{m=1}^\infty \frac{b_{n,m}^{p-r_n}}{m}, 
$$ 
and since 
$$
\limsup_{m\to\infty}\frac{\frac{b_{n,m+1}^{p-r_n}}{(m+1)}}{\frac{b_{n,m}^{p-r_n}}{m}}=
\limsup_{m\to\infty}\left(\frac{\mu(B_{n,m})}{\mu(B_{n,m+1})}\right)^{\frac{p-r_n}{r_n}} \leq 
\left(\frac12\right)^{\frac{p-r_n}{r_n}}<1,
$$ 
by ratio test for series we obtain that $g_n\in L^p$. 
Letting 
$$
g_B\doteq\sum_{n=1}^\infty \frac{g_n}{\|g_n\|2^n},
$$
we have that $g_B\in L^p$. It suffices to show now that $g_B\not\in L_q$ for any $q<p$. Fix such $q$; for a large enough $n$, $r_n>q$, and then 
\begin{align*}
(\|g_n\|2^n)^q\int |g_B|^q & \geq \int |g_n|^q  = 
\sum_m \left( \frac{1}{m.\mu(B_{n,m})}\right)^{\frac{q}{r_n}}\mu(B_{n,m})\\
& = \sum_m \left(\frac{1}{m}\right)^{\frac{q}{r_n}}\mu(B_{n,m})^{\frac{r_n-q}{r_n}}=
\mu(B_{n,1})^{\frac{r_n-q}{r_n}}\sum_m \left(\frac{1}{m}\right)^{\frac{q}{r_n}}=
\infty.
\end{align*}

\end{proof}

{\it Proof of Theorem \ref{mainth}(b).} Since $\mu$ is infinite and $\sigma$-finite, then each Borel set of infinite measure $D$ can be written as an infinite disjoint union of Borel sets of infinite measure. To see this it is enough to verify that $D$ \emph{contains} an infinite disjoint union of Borel sets of infinite measure $D_n$. In effect, we can define inductively Borel sets $C_k\subset D$ such that $1\leq \mu(C_k) <\infty$; let $(M_n)$ be a pairwise disjoint family of infinite subsets of $\N$. Then $D_n\doteq \bigcup_{k\in M_n} C_k$ is a family of Borel sets satisfying the desired properties. The desired result then follows from Lemma \ref{lemmapart2} and the same argument that was used in the proof of Theorem \ref{mainth}(a). \hfill{$\Box$}\\








The proof of Theorem \ref{mainth}(c) is a combination of the constructions from the proofs of parts (a) and (b):\\

{\it Proof of Theorem \ref{mainth}(c).} Consider $U_n$, $N_n$ and $f_m$ as in the proof of Theorem \ref{mainth}(a). Since $\mu(X\setminus\bigcup_n N_n)=\infty$, $X\setminus\bigcup_n N_n$ can be written as a disjoint union of Borel sets of infinite measure $D_m$. Then by Lemma \ref{lemmapart2}, for each $m$ there is a norm-one function $g_{D_m}\in L^p \setminus\bigcup_{q<p} L^q$ which is zero outside of $D_m$. Then the norm-one functions 
$$
\frac{f_m + g_{D_m}}{2},\, m\in\N
$$
are in $S_p'$ and have almost disjoint supports. \hfill{$\Box$}\\





\section{$V$-spaceability} 
\label{secvsp}

We have shown that, under special circumstances, the sets $S_p$, $L^p\setminus \cup_{0<q<p}L^q$ and $S'_p$, united to $\{0\}$, admit copies of $\ell_p$. This suggests the following definition: 

\begin{definition}

Let $V$ be a topological vector space and $S$ be a subset of $V$. Given a subspace $W$ of $V$, we say that $S$ is \emph{$W$-spaceable} if $S\cup \{0\}$ contains a $W$-isomorphic subspace of $V$.

\label{defvspc}
\end{definition}

There are plenty of examples where $V$-spaceability (of a subset of a topological vector space $V$) is a strictly more restrictive condition than maximal-dimension spaceability; for instance, $L^1[0,1]$ admits a subspace isomorphic to $\ell_2$, which turns to be maximal-dimension spaceable but not $L^1[0,1]$-spaceable. 
As usual, we can add adjectives like ``isometrically'', ``complementably'' and so on to ``$W$-spaceable'', depending on how nicely placed in $V$ is the copy of $W$ we have found in $S\cup \{0\}$. For example, Theorem \ref{mainth}(a) says that, under our assumptions, $S_p$ is isometrically (and complementably, if $p\geq 1$) $\ell_p$-spaceable. Note that, given a topological vector space $V$, the notion of $V$-spaceability of some subset $S$ of $V$ is quite strong; in particular, it implies that $S$ is maximal-dimension spaceable, and that $S\cup\{0\}$ contains copies of all subspaces of $V$. We have this phenomenon occurring, for example, for the set of nowhere differentiable functions in $C([0,1])$. The main theorem from \cite{r-p} can be reformulated as follows: 

\begin{theorem}[Rodr\'iguez-Piazza \cite{r-p}]

In $C([0,1])$, the set of nowhere differentiable functions is isometrically $C([0,1])$-spaceable. 

\end{theorem}

One step further is due to Hencl: 

\begin{theorem}[Hencl \cite{hencl}]

In $C([0,1])$, the set of nowhere approximatively differentiable and nowhere H\"older functions is isometrically $C([0,1])$-spaceable. 

\end{theorem}

Another example derives from Theorem \ref{mainth}(b): 

\begin{corollary}

In $\ell_p$, the set $\ell_p\setminus \cup_{0<q<p}\ell_q$ is isometrically $\ell_p$-spaceable, and if $p\geq 1$, it is isometrically and complementably $\ell_p$-spaceable. 

\end{corollary}

To see this, just notice that the positive integers with the counting measure satisfy the conditions in Theorem \ref{mainth}(b). This example is less interesting than the two previous ones, since all closed infinite-dimensional subspaces of $\ell_p$ are isomorphic to $\ell_p$, while $C([0,1])$ contains isometric copies of all separable Banach spaces. 

This remarks naturally motivate new directions of investigation concerning spaceability.  
In our context of nowhere $p$-integrable functions, we can pose the following:\\

\Problem Under appropriate assumptions, for which subspaces $V$ of $L^p$, is $S_p$ (or $L^p\setminus \cup_{0<q<p}L^q$, or $S'_p$) (isometrically, complementably...) $V$-spaceable? \\

The same could be asked when studying the spaceability of any other subset of a topological vector space. In the remaining of this section, we present some initial results in the direction of solving this problem (Theorems \ref{thl2} and \ref{thl201}, and Corollary \ref{coroll2}):

\begin{theorem}

Suppose that $1\leq p <\infty$ and that $(X,\mu)$ is a positive Borel measure space such that $S_p(X)$ is nonvoid. Then $S_p(X\times [0,1])$ is $\ell_2$-spaceable in $L^p(X\times [0,1])$. 

\label{thl2}
\end{theorem}

\begin{theorem}

$S_p([0,1])$ is $\ell_2$-spaceable in $L^p[0,1]$, for each $1\leq p <\infty$. 

\label{thl201}
\end{theorem}

\begin{corollary}

$S_p([0,1]^n)$ is $\ell_2$-spaceable in $L^p([0,1]^n)$, for each $1\leq p <\infty$ and each $n\in\N$. 

\label{coroll2}
\end{corollary}

Note that Corollary \ref{coroll2} follows easily by induction from Theorems \ref{thl2} and \ref{thl201}. Note also that Theorem \ref{thl201} is another improvement of Theorem \ref{quasi01} (for the $p\geq 1$ case), in a different direction if we compare to the improvement provided by Theorem \ref{mainth}(a).

To prove Theorem \ref{thl2}, we shall need some auxiliary results. The first one is a  corollary from the following: 

\begin{theorem}[Kitson, Timoney \cite{timoney}, Theorem 3.3]

Let $(E_n)$  be a sequence of Banach spaces and $F$ be a Fr\'echet space. Let $T_n:E_n\rightarrow F$ be continuous linear operators and $W\doteq span\{\bigcup_nT_n(E_n)\}$.

If $W$ is not closed in $F$, then $F\setminus W$ is spaceable.

\label{thKT}
\end{theorem}

\begin{corollary}

Let $E$ be an infinite-dimensional Banach space and $V$ be the linear span of a sequence of elements of $E$. Then $E\setminus V$ is spaceable. 

In particular, if $E$ is a sequence space, the set of elements $x=(x_n)$ of $E$ such that $x_n\neq 0$ for infinitely many $n$ is spaceable. 

\label{propspanspac}
\end{corollary}

\begin{proof} Note that, for the first part, it suffices to show that $E\setminus V$ is spaceable in the case where $E$ is the closed linear span of $(x_n)$ and $V$ is the linear span of $(x_n)$, for some linearly independent sequence $(x_n)$ in $E$. But in this case, defining $T_n:\R^n\rightarrow E$ by $T_n(\lambda_1,\dots,\lambda_n)\doteq \sum_{j=1}^{n}\lambda_jx_j$, we have that $V=span\{\bigcup_nT_n(\R^n)\}$. Since $V$ is not closed in $E$, we can apply Theorem \ref{thKT} and conclude the proof of the first part.

For the second part, just apply  the first part to $V=span\{e_n:\N\}$, where $\{e_n:\N\}$ is the canonical basis of $E$. 
\end{proof}


The sequence space we will be interested in will be $\ell_2$. Recall that, if $r_n$ are the Rademacher functions defined on $[0,1]$ by  $r_n(t) \doteq sign(\sin(2^n\pi t))$ and $0< p <\infty$, then $(r_n)$, as a sequence in $L^p[0,1]$, is equivalent to the canonical basis of $\ell_2$.

\begin{lemma}

Let $(a_n)$ be an element of $\ell_2$ having infinitely many nonzero $a_n$'s. Then  for all open $\emptyset\neq U\subset [0,1]$ we have that 
$(\textstyle{\sum} a_nr_n)|_U \not\equiv 0$, where $\textstyle{\sum} a_nr_n$ is a series in $L^p[0,1]$ $(0<p<\infty)$. 

\label{lemmarad}
\end{lemma}

\begin{proof} Let $(a_n)$ be an element of $\ell_2$ having infinitely many nonzero $a_n$'s, and let $U$ be a nonempty open subset of $[0,1]$. Let us denote 
$$
f\doteq \sum_{n=1}^\infty a_nr_n,\, f_{<j}\doteq \sum_{n=1}^{j-1} a_nr_n,\mbox{ and }
f_{\geq j} \doteq \sum_{n=j}^\infty a_nr_n.
$$
$U$ contains an interval of the form $I=[\frac{k}{2^{N}},\frac{k+1}{2^{N}}]$, for some $N\in\N$ and $k=0,\dots,N-1$. Note that $f_{<N}$ is constant in $I$, but since we have infinitely nonzero $a_n$'s, $f_{\geq N}$ is \emph{not} constant in $I$. Thus $f$ cannot be constant in $I$.
\end{proof}

\textit{Proof of Theorem \ref{thl2}.} By Corollary \ref{propspanspac}, there is a closed infinite-dimensional subspace $F$ of $\overline{span}\{r_n:n\in\N\}$ such that, for each nonzero element $\sum_{n=1}^\infty a_nr_n$ from $F$, we have that infinitely many elements from $(a_n)$ are nonzero. By Lemma \ref{lemmarad}, for each nonzero element $h$ of $F$ and each nonvoid open subset $U$ from $[0,1]$, we have $h|_U\not\equiv 0$. Note that, as a closed infinite-dimensional subspace of a space isomorphic to $\ell_2$, $F$ is isomorphic to $\ell_2$.

Let $f$ be a norm-one element from $S_p(X)$, and define $\Phi:F\rightarrow L^p(X\times [0,1])$ by
$$
\Phi\left(\sum_{n=1}^\infty a_n r_n\right)(x,t)\doteq f(x) \sum_{n=1}^\infty a_n r_n(t).
$$
Note that the support of $\Phi(h)$ is $\sigma$-finite for each $h\in F$, and $\Phi$ is clearly an isometric isomorphism onto its range. By Fubini's theorem, 
\begin{align*}
\left\|\Phi\left(\sum_{n=1}^\infty a_n r_n\right)\right\|_p^p &= \int_X\int_0^1\left|f(x) \sum_{n=1}^\infty a_n r_n(t)\right|^p dt\,dx\\
&= \int_0^1\left(\int_X |f(x)|^p\, dx\right)\left|\sum_{n=1}^\infty a_n r_n(t)\right|^p dt\\
& =\int_0^1\left|\sum_{n=1}^\infty a_n r_n(t)\right|^p dt = \left\|\sum_{n=1}^\infty a_n r_n\right\|_p^p,
\end{align*}
for all $\sum_{n=1}^\infty a_n r_n\in F$. It follows that $\Phi(F)$ is a $\ell_2$-isomorphic subspace from $L^p(X\times [0,1])$. 

It remains to show that $\Phi(F)\subset S_p(X\times [0,1])\cup \{0\}$. Let $\sum_{n=1}^\infty a_n r_n$ be a nonzero element from $F$, $U\times (a,b)$ be a nonvoid basic open subset of $X\times [0,1]$, and $p\leq q <\infty$. Then 
$$
\int_{U\times (a,b)}\left|\Phi\left(\sum_{n=1}^\infty a_n r_n\right)\right|^q = 
\int_a^b \left|\sum_{n=1}^\infty a_nr_n(t)\right|^qdt \int_U|f(x)|^qdx
$$ 
converges if $q=p$ (by Khinchine's inequality and since $f$ is $p$-integrable), and does not converge if $q>p$ (since the first factor is strictly positive by Lemma \ref{lemmarad} and $f$ is not $q$-integrable). This concludes our proof. 
\hfill{$\Box$}\\

Before proceeding to the proof of Theorem \ref{thl201}, we point out that it is not hard to prove, using Theorem \ref{thl2}, that $L^p[0,1]\setminus \bigcup_{q>p}L^q[0,1]$ is $\ell_2$-spaceable in $L^p[0,1]$, for $1\leq p <\infty$. In effect, recall that there exists a measure-preserving Borel isomorphism $\psi$ from $[0,1]^2$ onto $[0,1]$, which in turn induces an isometric isomorphism  $\Psi$ from $L^p([0,1]^2)$ onto $L^p[0,1]$, defined by $\Psi(f)\doteq f\circ \psi^{-1}$. It is easy to verify that, for each $f\in L^p([0,1]^2)$ and each $q>p$, $f$ is $q$-integrable if and only if $\Psi(f)$ is $q$-integrable; in particular, all functions in $\Psi(S_p([0,1]^2))$ are \emph{not} $q$-integrable for $q>p$, and the claim follows. But the \emph{nowhere} part is lost, since $\psi$ is not an homeomorphism. We need thus to provide a finer construction.\\

\textit{Proof of Theorem \ref{thl201}.} Let $\mu$ be the Lebesgue measure on $[0,1]$, that is, the unique Borel measure such that $\mu([k/2^n,(k+1)/2^n])=1/2^n$ for every $n\in\N$ and $k=0,1,\dots,2^n-1$. By $\lambda$ denote the Lebesgue measure on $\{0,1\}^\N$, that is, the unique Borel measure such that $\lambda(\langle s\rangle)=1/2^{|s|}$ for every finite sequence $s$ of zeros and ones where $|s|$ is the length of $s$ and $\langle s\rangle$ stands for the set of all $x\in\{0,1\}^\N$ such that $x(k)=s(k)$ for $k=1,2,\dots,|s|$. Let $g:\{0,1\}^\N\to[0,1]$ be given by
$g(x)=\sum_{n=1}^\infty\frac{x(n)}{2^n}.$ Note that $g^{-1}(k/2^n)$ consists of two elements $x,y$ such that $x$ is a binary representation of $k/2^n$ with $x(m)=0$ for $m>n$, and $y$ is a binary representation of $k/2^n$ with $x(m)=1$ for $m>n$. Moreover $g^{-1}(t)$ is a singleton if $t$ is not of the form $k/2^n$. 

\begin{claim}
$\mu(A)=\lambda(g^{-1}(A))$ for every Borel set $A$ in $[0,1]$ and 
$\lambda(B)=\mu(g(B))$ for every Borel set $B$ in $\{0,1\}^\N$. 
\end{claim}

It is enough to show this for $A=[k/2^n,(k+1)/2^n]$ and $B=\langle s\rangle$. Let $s$ be a finite sequence of zeros and ones which is a binary representation of the number $k/2^n$. Then $|s|=n$ and $g^{-1}(A)=\langle s\rangle$. Thus $\lambda(g^{-1}(A))=\lambda(\langle s\rangle)=1/2^n=\mu(A)$.  

Let $n=|s|$ and define $k=s(n)+2s(n-1)+2^2s(n-2)+\dots+2^{n-1}s(1)$. Then $g(B)=g(\langle s\rangle)=[k/2^n,(k+1)/2^n]$. Thus $\mu(g(B))=\mu([k/2^n,(k+1)/2^n])=1/2^n=1/2^{|s|}=\lambda(B)$. Claim 1 is then proved. 

\begin{claim}
$F:L^p[0,1]\rightarrow L^p(\{0,1\}^\N)$ given by $F(f)\doteq f\circ g$ is an isometric isomorphism between $L^p[0,1]$ and $L^p(\{0,1\}^\N)$, with inverse $L:L^p(\{0,1\}^\N)\rightarrow L^p[0,1]$ given by 
$$
L(h)(t)\doteq \left\{ 
\begin{array}{l}
h(g^{-1}(t)), \mbox{if $t\in [0,1]$ is not of the form $k/2^n$};\\
h(x), \mbox{\small if $t=k/2^n$ and $x$ is the binary representation of $t$ with $x(m)=0,m>n$.}
\end{array} \right.
$$
Moreover, $F|_{L^q[0,1]}= L^q(\{0,1\}^\N)$ for each $q>p$. 
\end{claim}

If $A$ is a Borel set in $[0,1]$ and $f=\chi_A$ is the characteristic function of a set $A$, then by Claim 1 we have
\begin{eqnarray}
\int_{[0,1]}f d\mu=\mu(A)=\lambda(g^{-1}(A))=\int_{\{0,1\}^\N} \chi_{g^{-1}(A)} d\lambda=\int_{\{0,1\}^\N} f\circ g d\lambda.
\label{gggg}
\end{eqnarray}
It is easily seen that (\ref{gggg}) also holds if $f$ is a step function, and it follows that 
$$
\int_{[0,1]}\vert f\vert d\mu=\int_{\{0,1\}^\N}\vert f\circ g\vert d\lambda.
$$ 
holds for each $f\in L^1[0,1]$. It follows easily that $\Vert f\Vert_p=\Vert f\circ g\Vert_p$, for $f\in L^p[0,1]$. This shows that $F$ is norm-preserving and that $F|_{L^q[0,1]}= L^q(\{0,1\}^\N)$ for each $q>p$. 

It is clear that $L$ is a left inverse for $F$. Note that, for a given $h\in L^p(\{0,1\}^\N)$,  $F(L(h))$ eventually differs from $h$ on a countable set of elements $x\in\{0,1\}^\N$ with $x(m)=1$ for almost every $m$. Since $\lambda$ is a continuous measure, then $\lambda(\{x\in\{0,1\}^\N:h(x)\neq F(L(h))(x)\})=0$. This means that $h$ and $F(L(h))$ are the same element of $L^p(\{0,1\}^\N)$.  Thus Claim 2 is proved.

\begin{claim}
$F(S_p([0,1]))=S_p(\{0,1\}^\N)$ and $L(S_p(\{0,1\}^\N))=S_p([0,1])$.
\end{claim}

Let $f\in S_p([0,1])$ and fix a basic set $\langle s\rangle$ in $\{0,1\}^\N$. Note that there exists a positive integer $k$ such that $g(\langle s \rangle )= [k/2^n,(k+1)/2^n]$, where $n=|s|$. Since 
$$
F(f)\chi_{\langle s \rangle}= (f\circ g)\chi_{\langle s \rangle} = (f\chi_{[k/2^n,(k+1)/2^n]})\circ g 
= F(f\chi_{[k/2^n,(k+1)/2^n]})
$$
and $f\chi_{[k/2^n,(k+1)/2^n}\not\in L^q[0,1]$, it follows by the previous claim that $F(f)\chi_{\langle s \rangle}\not\in L^q(\{0,1\}^\N)$. That means that $F(f)$ is nowhere $L^q$. We have then proved that $F(S_p([0,1]))\subset S_p(\{0,1\}^\N)$ and, using the previous Claim, that $L(S_p(\{0,1\}^\N))\supset S_p([0,1])$.

Now let $h\in S_p(\{0,1\}^\N)$ and fix a set $[k/2^n,(k+1)/2^n]$. Let $s$ be a finite set which is a binary representation of $k/2^n$. Since $h\chi_{\langle s\rangle}$ is not in $L^q$ for $q>p$, then 
$$
L(h\chi_{\langle s\rangle})=(h\circ g^{-1})\chi_{[k/2^n,(k+1)/2^n]}=L(h)\chi_{[k/2^n,(k+1)/2^n]}.
$$ Thus $L(h)$ is nowhere $L^q$ and the proof of Claim 3 is complete. \\

Given two positive Borel measure spaces $X$ and $Y$ and $0<p<\infty$, we shall say that an application $\varphi:L^p(X)\to L^p(Y)$ \emph{preserves $S_p$} if $\varphi (S_p(X))\subset S_p(Y)$. Claim 3 asserts in particular that both $F$ and $L$ preserve $S_p$.

\begin{claim}
There is an isometric isomorphism $G$ from $L^p(\{0,1\}^\N)$ onto $L^p(\{0,1\}^\N\times\{0,1\}^\N)$ which preserves $S_p$. 
\end{claim}

Let $\varphi:\{0,1\}^\N\times\{0,1\}^\N\to\{0,1\}^\N$ be defined by
$$
\varphi((x(1),x(2),x(3),\dots),(y(1),y(2),y(3),\dots))\doteq (x(1),y(1),x(2),y(2),x(3),y(3),\dots).
$$
It is well known that $\varphi$ is a homeomorphism of $\{0,1\}^\N\times\{0,1\}^\N$ and  $\{0,1\}^\N$. Fix two finite sequences $s$ and $s'$ of zeros and ones. Note that 
$$
\frac{1}{2^{|s|}}\frac{1}{2^{|s'|}}=\lambda(\langle s\rangle)\lambda(\langle s'\rangle)=\lambda\times\lambda(\langle s\rangle\times \langle s'\rangle)
$$
and 
$$
\lambda(\varphi(\langle s\rangle\times \langle s'\rangle))=\frac{1}{2^{|s|+|s'|}}.
$$
The last equality follows from the fact that the set $\varphi(\langle s\rangle\times \langle s'\rangle)$ is a subset of $\{0,1\}^\N$ such that exactly $|s|+|s'|$ of its coordinates are fixed. 

Using this we obtain that $\lambda\times\lambda(A)=\lambda(\varphi(A))$ for any Borel subset $A$ of $\{0,1\}^\N\times\{0,1\}^\N$. Then $G:L^p(\{0,1\}^\N\times\{0,1\}^\N)\to L^p(\{0,1\}^\N)$ defined by $G(f)\doteq f\circ\varphi$ has inverse given by $G^{-1}(h)=h\circ\varphi^{-1}$ and satisfies the desired properties. This completes the proof of Claim 4. \\

Mimicking the reasoning used to prove Claims 1--3, we can show that $T:L^p( \{0,1\}^\N\times [0,1])\to L^p(\{0,1\}^\N\times\{0,1\}^\N)$ defined by 
$$
T(f)(x,y)\doteq f(x,g(y))
$$
is an onto isometric isomorphism which preserves $S_p$. Hence, we have built the following chain of $S_p$-preserving isometric isomorphisms:
$$
L^p(\{0,1\}^\N\!\times \![0,1])\stackrel{T}{\longrightarrow} L^p(\{0,1\}^\N\!\times\!\{0,1\}^\N)
\stackrel{G}{\longrightarrow} L^p(\{0,1\}^\N)
\stackrel{L}{\longrightarrow} L^p[0,1].
$$
Theorem \ref{mainth} implies that $S_p(\{0,1\}^\N)$ is nonempty, and from Theorem \ref{thl2} we then obtain that $S_p(\{0,1\}^\N\times[0,1])$ is $\ell_2$-spaceable in $L^p(\{0,1\}^\N\times [0,1])$. The conclusion follows immediately. \hfill{$\Box$}


\section{Dense-lineability: proof of Theorem \ref{mainth}(d)}
\label{secdens}

We start by establishing some notation before proceeding to the proof of Theorem \ref{mainth}(d), which will also be via Lemma \ref{fact2}. First, recall that  a family $\{A_i:i\in I\}$ of infinite subsets of $\N$ is said to be \emph{almost disjoint} if $A_i\cap A_j$ is finite for any distinct $i,j\in I$. It is well known that there is a family of almost disjoint subsets of $\N$ of cardinality continuum. Let $\{A'_\alpha:\alpha<\mathfrak{c}\}$ be such family. 

Fix a sequence of integers $1=n_0<n_1<n_2<n_3<\dots$  such that 
$$
\sum_{i=n_k}^{n_{k+1}-1}\frac1i\geq 1,
$$ 
and consider $M_k\doteq\{n_k,n_k+1,\dots,n_{k+1}-1\}$.  Define, for each $\alpha < \mathfrak{c}$, $A_\alpha\doteq\bigcup\{M_k:k\in A'_\alpha\}$. Note that $\{A_\alpha:\alpha<\mathfrak{c}\}$ is an almost disjoint family and that 
$$
\sum_{i\in A_\alpha}\frac1i=\infty 
$$
for each $\alpha < \mathfrak{c}$. We shall fix the family  $\{A_\alpha:\alpha<\mathfrak{c}\}$ and use it in the following. \\

\textit{Proof of Theorem \ref{mainth}(d). } For a fixed Borel set $A$ of positive finite measure and  $\alpha<\mathfrak{c}$ we define a function  $h^\alpha_A$ as follows.  Let $\{A_{n,m}:n,m\in\N\}$ be a family of pairwise disjoint subsets of $A$ of positive measure such that $\mu(A_{n,m})\geq 2\mu(A_{n,m+1})$, and let $(r_n)$ be a strictly decreasing sequence of reals tending to $p$. Put 
\begin{eqnarray}
h^\alpha_n\doteq \sum_{m\in A_\alpha} a_{m,n}\chi_{A_{n,m}},
\label{halpha_n}
\end{eqnarray}
where $a^{r_n}_{m,n}\mu(A_{n,m})=1/m$. Then a similar argument as used in Lemma \ref{hA} leads us to that the $A$-supported, norm-one function
\begin{eqnarray}
h^\alpha_A\doteq\sum_{n=1}^\infty \frac{h^\alpha_n}{\|h^\alpha_n\|2^n}
\label{halphaA}
\end{eqnarray}
is in $L^p\setminus\bigcup_{q>p}L^q$.

As in the proof of Theorem \ref{fact1}, fix a basis $(U_n)$ for $X$, and let $N_n\subset U_n$ be a sequence of pairwise disjoint nowhere dense Borel sets satisfying $0<\mu(N_n)<1/2^n$. For each $\alpha<\mathfrak{c}$, by defining $h^\alpha_{N_{n}}$ as in (\ref{halphaA}) and putting
$$
f^\alpha\doteq\sum_{n=1}^\infty \frac{h^\alpha_{N_{n}}}{2^n}, 
$$
we obtain that  $f^\alpha\in S_p$ and has norm one.

Note that any ordinal number $\alpha<\mathfrak{c}$ is of the form $\beta+n$, where $\beta$ is a limit ordinal and $n=0,1,2,\dots$. Let $\{B_\beta:\beta<\mathfrak{c}\}$ be an indexation of all Borel subsets of $X$. Then the set   $\{(B_\beta,n):\beta<\mathfrak{c},n\in\N\}$ has cardinality $\mathfrak{c}$,
thus there is a bijection $(B_\beta,n)\mapsto \alpha(\beta,n)$ onto all ordinals less than $\mathfrak{c}$. 

Consider, for $\beta<\mathfrak{c}$ and $n\in\N$, the functions 
\begin{align}
g^{\beta,n}\doteq g^{\alpha(\beta,n)}\doteq\chi_{B_\beta}+\frac1n f^{\alpha(\beta,n)}.
\label{gbetan}
\end{align}

By our construction,  the linear span of $\{g^{\alpha(\beta,n)}:\beta<\mathfrak{c}, n\in\N\}$ is dense in the set of all simple functions on $X$, and therefore it is also dense in $L^p$. We will show show that any nontrivial linear combination of functions of the form (\ref{gbetan}) is in $S_p$. 
Let $(\beta_1,n_1),\dots,(\beta_k,n_k)$ be distinct and consider $b_1,\dots,b_k\in\R$ which are not all zero, and write
$$
g\doteq b_1 g^{\beta_1,n_1}+\dots +b_k g^{\beta_k,n_k}=(b_1\chi_{B_{\beta_1}}+\dots + b_k\chi_{B_{\beta_k}})
+\frac{b_1}{n_1}f^{\alpha(\beta_1,n_1)}+\dots +\frac{b_k}{n_k}f^{\alpha(\beta_k,n_k)}.
$$
Consider $\alpha_i\doteq \alpha(\beta_i,n_i)$, and note that $\alpha_1,\dots,\alpha_k$ are distinct ordinal numbers. We can then write 
\begin{align*}
g &= (b_1\chi_{B_{\beta_1}}+\dots + b_k\chi_{B_{\beta_k}})+\frac{b_1}{n_1}f^{\alpha_1}+\dots +\frac{b_k}{n_k}f^{\alpha_k}\\
&= (b_1\chi_{B_{\beta_1}}+\dots +b_k\chi_{B_{\beta_k}})
+\frac{b_1}{n_1}\sum_{n=1}^\infty \frac{h^{\alpha_1}_{N_{n}}}{2^n}+\dots +\frac{b_k}{n_k}\sum_{n=1}^\infty \frac{h^{\alpha_k}_{N_{n}}}{2^n}.
\end{align*}

Consider the family $\{A_{l,m}:l,m\in\N\}$ of pairwise disjoint subsets of $N_{n}$ of positive measure such that $\mu(A_{l,m})\geq 2\mu(A_{l,m+1})$, and construct $h^{\alpha_i}_{N_{n}}$ as in (\ref{halpha_n}), using these sets and the corresponding $a_{l,m}$. Consider $N\in\N$ such that the sets $C_1\doteq A_{\alpha_1}\setminus\{1,2,\dots,N\}, C_2\doteq A_{\alpha_2}\setminus\{1,2,\dots,N\},\dots,C_k\doteq A_{\alpha_k}\setminus\{1,2,\dots,N\}$ are disjoint; this is possible since $\{A_\alpha:\alpha<\mathfrak{c}\}$ is almost disjoint. Then we have
$$
h^{\alpha_i}_l=\sum_{m\in A_{\alpha_i}} a_m\chi_{A_{l,m}}=\sum_{m\in A_{\alpha_i}\bigcap\{1,\dots,N\}} a_m\chi_{A_{l,m}}+\sum_{m\in C_i} a_m\chi_{A_{l,m}},
$$    
and thus 
\begin{align*}
h^{\alpha_i}_{N_{n}}&=\sum_{l=1}^\infty \frac{h^{\alpha_i}_l}{||h^{\alpha_i}_l||2^l}=
\sum_{l=1}^\infty \frac{1}{||h^{\alpha_i}_l||2^l}\left(\sum_{m\in A_{\alpha_i}\bigcap\{1,...,N\}} a_m\chi_{A_{l,m}}+\sum_{m\in C_i} a_m\chi_{A_{l,m}}\right)\\
&= \sum_{l=1}^\infty \frac{1}{||h^{\alpha_i}_l||2^l}\sum_{m\in A_{\alpha_i}\bigcap\{1,...,N\}} a_m\chi_{A_{l,m}}+\sum_{l=1}^\infty \frac{1}{||h^{\alpha_i}_l||2^l}\sum_{m\in C_i} a_m\chi_{A_{l,m}}.
\end{align*}

Writing $w_i\doteq \sum_{l=1}^\infty \frac{1}{||h^{\alpha_i}_l||2^l}\sum_{m\in C_i} a_m\chi_{A_{l,m}}$ for each $i=1,\dots,k$,  
by our construction we have that each $w_i$ is in $L_p\setminus\bigcup_{q>p}L_q$ and $w_1,\dots,w_k$ have disjoint supports; more precisely, the support of each $w_i$ is $N_n^i\doteq\bigcup_{m}\bigcup_{l\in C_i}A_{l,m}$.  Note that  $span \{f^{\alpha_i}\chi_{\bigcup_{n}N_n^i}\}\subset S_p$. The fact that $g\in S_p$ follows then from the fact that adding a simple function to a function from $S_p$ results in a function from $S_p$. 

Since $L^p(X,\mu)$ is separable, it has dimension $\mathfrak{c}$, as does $span\{g^{\alpha(\beta,n)}:\beta<\mathfrak{c}, n\in\N\}$, which concludes our proof. \hfill{$\Box$}


\section{Algebrability: proof of Theorem \ref{mainth}(e)}
\label{secalg}

\textit{Proof of Theorem \ref{mainth}(e)} Let $(U_n)$ be a basis for $X$. Similarly to the construction held at the beginning of the Proof of Theorem \ref{mainth}(a), one can find pairwise disjoint nowhere dense Borel sets $N_n$ such that $N_n\subset U_n$ and $0<\mu(N_n)<\frac{1}{2^n}$. Using Lemma \ref{fact2}, we can find for each $n$ a pairwise disjoint family $(N_{n,j})_j$ of Borel subsets of $N_n$ satisfying 
$$
\mu(N_{n,j+1})\leq \frac{1}{j+1}\mu(N_{n,j}).
$$
Note that, for each $n,j$, we have $\mu(N_{n,j})\leq \frac{1}{j!2^n}$. Let $B_j\doteq \bigcup_n N_{n,j}$. Then all nonvoid open subsets of $X$ intercept each $B_j$ in non-null sets, and by the other hand $\mu(B_j)=\sum_n \mu(N_{n,j}) \leq \frac{1}{j!}$. 

Let $\{\theta_\alpha:\alpha<\mathfrak{c}\}$ be a set of real numbers strictly greater than $1$ such that the set $\{\ln(\theta_\alpha):\alpha<\mathfrak{c}\}$ is linearly independent over the rational numbers. For each $\alpha<\mathfrak{c}$, define 
$$
g_\alpha\doteq \sum_{j=1}^\infty\theta^j_\alpha\chi_{B_j}.
$$
For each $\alpha$ the series $\sum_j \frac{\theta_\alpha^{pj}}{j!}$ converges, thus each $g_\alpha\in L^p$, for each $\alpha<\mathfrak{c}$ and each $0<p<\infty$. 

Let us show that $\{g_\alpha:\alpha<\mathfrak{c}\}$ is a set of free generators, and the algebra generated by this set is contained in $\mathcal{G}\cup\{0\}$. It suffices to show that, for every $m$ and $n$ positive integers, for every matrix $(k_{il}:i=1,\dots, m,\, l=1,\dots, n)$ of non-negative integers with non-zero and distinct rows, for every $\alpha_1,\dots,\alpha_n <\mathfrak{c}$ and for every $\beta_1,\dots,\beta_m\in\R$ which do not vanish simultaneously, the function
\begin{align*}
g & \doteq\beta_1g_{\alpha_1}^{k_{11}}\dots g_{\alpha_n}^{k_{1n}}+\dots+ \beta_mg_{\alpha_1}^{k_{m1}}\dots g_{\alpha_n}^{k_{mn}}\\
& =
\sum_{j=1}^\infty(\beta_1(\theta_{\alpha_{1}}^{k_{11}}\dots\theta_{\alpha_{n}}^{k_{1n}})^j+\dots+ \beta_m(\theta_{\alpha_{1}}^{k_{m1}}\dots\theta_{\alpha_{n}}^{k_{mn}})^j)\chi_{B_j}
\end{align*}
is in $\mathcal{G}$. First, let us show that it is in $\bigcap_{0<p<\infty} L^p$. Fix $p$ and put, for each $i=1,\dots,m$, $\theta_i\doteq \theta_{\alpha_1}^{k_{i1}}\cdots \theta_{\alpha_n}^{k_{in}}$. Then
\begin{eqnarray}
\int |g|^p  \leq 
\int \left[\sum_{j=1}^\infty(|\beta_1|\theta_1^j+\dots+ |\beta_m|\theta_m^j)^p\chi_{B_j}\right] \leq
\sum_{j=1}^\infty\frac{Q(\theta_1^j,\dots,\theta_m^j)}{j!}, 
\label{qpwo}
\end{eqnarray}
where $Q:(x_1,\dots,x_m)\mapsto (|\beta_1|\theta_1^j+\dots+ |\beta_m|\theta_m^j)^p$. It is straightforward to find $C,b>0$ such that $Q(\theta_1^j,\dots,\theta_m^j)<C+b^j$ for all $j$. Thus the sum on the right handside of (\ref{qpwo}) converges and $g\in L^p$.

Since $\ln(\theta_i)=\ln(\theta_{\alpha_1}^{k_{i1}}\cdots \theta_{\alpha_n}^{k_{in}})=
k_{i1}\ln\theta_{\alpha_1}+\dots+k_{in}\ln\theta_{\alpha_n}$
and $\ln\theta_{\alpha_1},\dots,\ln\theta_{\alpha_n}$ are $\mathbb{Q}$-linearly independent, the numbers $\ln(\theta_1),\dots,\ln(\theta_m)$ are distinct. Then by the strict monotonicity of the logarithmic function we may assume that 
\begin{eqnarray}
\theta_1>\dots >\theta_m;
\label{thetas}
\end{eqnarray}
we also may assume $\beta_1\neq 0$. Then we can write
$$
g=\sum_{j=1}^\infty(\beta_1\theta^j_1+\dots+\beta_m\theta^j_m)\chi_{B_j}.
$$
From (\ref{thetas}) and since $\beta_1$ is assumed to be nonzero, we can find $j_0\in\mathbb{N}$ such that 
$$
|\beta_2|\theta^j_2+\dots+|\beta_m|\theta^j_m<\frac12 |\beta_1|\theta^j_1
$$
for all $j\geq j_0$. Then for those $j$
\begin{align*}
|\beta_1\theta^j_1+\dots+\beta_m\theta^j_m| &\geq |\beta_1|\theta^j_1 - \left| \beta_2 \theta^j_2 +\dots+ \beta_m\theta^j_m\right|\\
&\geq |\beta_1|\theta^j_1 - \left( |\beta_2| \theta^j_2 +...+ |\beta_m| \theta^j_m\right) >  \frac12 |\beta_1|\theta^j_1. 
\end{align*}
Since each nonvoid open subset of $X$ intercepts all $B_j$ in non-null sets, the inequality above shows that $g$ is nowhere essentially bounded. \hfill{$\Box$}

\subsection{Comments and open problems} 

As a Corollary from Theorem \ref{mainth}(e) we have the following:

\begin{corollary}

If $\mu$ is an atomless and outer regular positive Borel measure on $X$ with full support and $0<p<\infty$, then
$$
\mathcal{G}_p\doteq \left\{f\in L^p(\mu): \mbox{ $f$ is nowhere }L^\infty(\mu)\right\}
$$
is strongly $\mathfrak{c}$-algebrable. 

\end{corollary}

It is a straightforward exercise for the reader to show, using a construction similar to the one used to prove Theorem \ref{mainth}(a), that $\mathcal{G}_p$ is spaceable in $L^p$. To finish this section we pose the following problem: \\

\Problem Does $\mathcal{G}_p\cup\{0\}$ admit \emph{dense} or \emph{closed} subalgebras of $L^p$?




\section{When are there nowhere $q$-integrable functions in $L^p$?}
\label{secfinal}

We conclude this work with a couple of remarks and questions on necessary/sufficient conditions on a positive Borel measure space $(X,\mu)$, so that there exist nowhere $q$-integrable Borel functions in $L^p(X)$. An obvious necessary condition is that $\mu$ has full support, so we will always assume that. It is not hard to see that it is also necessary that $X$ has the countable chain condition, as Proposition \ref{propnecccc} below shows. Recall that $X$ is said to have the countable chain condition (or ccc, in short) if any family consisting of open non-empty pairwise disjoint subsets of $X$ is countable. 

\begin{proposition}
Let $X$ be a topological space without the ccc, assume that $\mu$ positive Borel measure on $X$ with full support, fix $0<q<\infty$ and let $f:X\to\R$ be a Borel function. 

If $f\vert_U$ is not in $L^q(U)$ for any nonvoid open set $U$, then $f$ is not in $L^p(X)$ for any $0<p<\infty$.  
\label{propnecccc}
\end{proposition}  

\begin{proof}
Let $(U_s)_{s\in S}$ be an uncountable family of pairwise disjoint non-empty open sets. Since $f\vert_{U_s}$ is not in $L^q(U_s)$ for any $U_s$, then $f$ does not vanish on $U_s$. Thus for each $0<p<\infty$ and each $s\in S$,  $\Vert f|_{U_s}\Vert_p>0$. Fix $0<p<\infty$. Since $S$ is uncountable, then at least one of the sets  
$$
S_n\doteq\left\{s\in S: \int_{U_s}|f|^pd\mu\geq\frac1n\right\}
$$  
is uncountable. Hence $\int_X|f|^pd\mu=\infty$.
\end{proof}

The next natural step is pose the following: \\

\Problem Suppose that $X$ has the ccc, and that $\mu$ is a positive Borel measure on $X$ with full support. Given $0<p<\infty$, does there exist a Borel $p$-integrable function $f:X\to\R$ which is nowhere $q$-integrable for $q>p$?\\

We provide a partial answer to the problem above, through a consistency result. Recall first that the product of two spaces with the ccc do not need to have the ccc, however this statement is independent of ZFC. Under Martin's axiom, the product of two ccc spaces has the ccc, but in some models of ZFC there exists a topological space called \emph{Suslin line}, which has the ccc but its square does not have the ccc (cf. \cite{kunen}).

\begin{theorem}
It is consistent with ZFC that there is a topological space $X$ satisfying the ccc such that, for any positive Borel measure $\mu$ on $X$ with full support and any $0<p<\infty$, there is no Borel function $f:X\to\R$ in $L^p(\mu)$ but nowhere $L^q(\mu)$ for $q>p$.
\label{consistth}
\end{theorem}

\begin{proof}
It is consistent with ZFC that there exists a Suslin line $X$. Suppose that there is a Borel function $f:X\to\R$ in $L^p(\mu)$ but nowhere $L^q(\mu)$ for $q>p$. Let $\tilde{f}:X^2\to\R$ be defined by $\tilde{f}(x,y)=f(x)f(y)$. Clearly $\tilde{f}$ is Borel, and since $supp\, f$ is $\sigma$-finite, so is $supp\,\tilde{f}=(supp\,f)^2$. Fubini's theorem then implies that 
$$
\int_{X^2}\vert\tilde{f}\vert^pd(\mu\!\times\!\mu)=\Vert f\Vert_p^{2p},
$$ 
and for any two nonvoid open sets $U,V\subset X$  and $q>p$ we have that
$$
\int_{U\times V}\vert\tilde{f}\vert^qd(\mu\!\times\!\mu)=\int_U\vert f\vert^q\int_V\vert f\vert^q=\infty.
$$
Hence $\tilde{f}$ is a Borel function in $L^p(\mu\times\mu)$ but nowhere $L^q(\mu\times\mu)$ for $q>p$ and $\mu\times\mu$ is a positive Borel measure with full support. Since $X^2$ does not have the ccc, we get a contradiction.
\end{proof}

Finally we turn our attention to the presence of countable $\pi$-bases. First, note that there exist topological spaces $X$ with countable $\pi$-bases but admitting  no countable bases, and with positive Borel measures with full support defined on them: take for example the Sorgenfrey line (the set of real numbers with the topology generated by intervals of the form $[a,b)$) with the Lebesgue measure. The Sorgenfrey line $\R_S$ has a countable $\pi$-basis thus we can apply Theorem \ref{mainth} to show that $S_p(\R_S)$ is $\ell_p$-spaceable, but any basis  of the Sorgenfrey line has cardinality  $\mathfrak{c}$. 

It turns out that the presence of a countable $\pi$-basis in $X$ is also not necessary for the existence of nowhere $q$-integrable functions in $L^p(X)$. In fact, we have more.

\begin{theorem}
Let $X$ be a topological space with a countable $\pi$-basis. 
Suppose that $\mu$ is an atomless and outer-regular Borel probability measure with full support. Assume that $\kappa$ is an uncountable cardinal number. Let $Y=X^\kappa$ be the Tychonoff product of $\kappa$ many copies of $X$, and let $\lambda=\mu^\kappa$ be the product of $\kappa$ many copies of $\mu$. Then the $S_p(Y)$ is spaceable in $Y$.   
\end{theorem}

\begin{proof}
Let $(U_n)$ be a countable $\pi$-base in $X$.  By a construction used to prove Theorem \ref{mainth}(a) applied to $X$ and $\mu$, there is a norm-one basic sequence $f_1,f_2,\dots$ of elements of $S_p(X)$ with pairwise disjoint
supports, and each $f_i$ is of the form $\sum_{k=1}^\infty a_k\chi_{A_k}$ where $A_k$ are Borel subsets of $X$. For $(x_\alpha)_{\alpha<\kappa}\in X^\kappa$, put $\tilde{f}_i((x_\alpha)_{\alpha<\kappa})\doteq f_i(x_0)$. Then $(\tilde{f}_i)$ is a norm-one basic sequence with pairwise disjoint
supports. We need to show that each of them is in $S_p(Y)$. Note that $\tilde{f}_i$ is of the form $\sum_{k=1}^\infty a_k\chi_{\tilde{A}_k}$, where $\tilde{A}_k\doteq A_k\times\prod_{1\leq\alpha<\kappa}X$.

Let $V$ be an nonempty open subset of $Y$. We may assume that $V$ is of the form $\prod_{\alpha<\kappa}W_\alpha$ where $W_\alpha$ are nonempty open subsets of $X$ and there is finite set $F\subset\kappa$ such that $W_\alpha=X$ if $\alpha\in\kappa\setminus F$. Let $F_0=F\setminus\{0\}$. We have $V=W_0\times\prod_{1\leq\alpha<\kappa}W_\alpha$ and 
\begin{align*}
\int_{V}|\tilde{f}_i|^qd\lambda &=\int_{V}\sum_{k=1}^\infty |a_k|^q\chi_{\tilde{A}_k}d\lambda=
\sum_{k=1}^\infty |a_k|^q\lambda(\tilde{A}_k\cap V)\\
&=\sum_{k=1}^\infty |a_k|^q\lambda\left(({A}_k\cap W_0)\times\prod_{1\leq\alpha<\kappa}W_\alpha\right)
=\sum_{k=1}^\infty |a_k|^q\mu({A}_k\cap W_0)\prod_{\alpha\in F_0}\mu(W_\alpha)\\
&=\prod_{\alpha\in F_0}\mu(W_\alpha)\int_{W_0}\vert f_i\vert^q d\mu=\infty.
\end{align*}
Similarly we get that
$$
\int_{Y}\vert\tilde{f}_i\vert^pd\lambda=\int_X\vert f\vert^pd\mu<\infty,
$$
and our proof is concluded.
\end{proof}



\textbf{Acknowledgement:} we would like to thank prof. Gilles Godefroy for several comments that lead to improvements of this work.

\end{document}